 \documentclass[10pt,twocolumn,twoside]{IEEEtran}

\usepackage{scrextend}

\usepackage{etoolbox}

\newbool{proof_diagonal_stability}
\boolfalse{proof_diagonal_stability}

\usepackage{bbm}

\usepackage{xcolor}
\usepackage{cite}
\colorlet{myred}{black}
\colorlet{extensions}{blue!45!red!100!orange!90!}

\usepackage[english]{babel}
\usepackage{subcaption}
\usepackage{graphicx}
\usepackage{amsthm}
\usepackage{mathtools}
\usepackage{breqn}
\usepackage{amssymb}
\usepackage{dsfont}

\usepackage{euscript}

\newcommand{\norm}[1]{\left\lVert#1\right\rVert}

\newtheorem*{remark}{Remark}
\RequirePackage{fix-cm} % arbitrary font scaling
\usepackage{bbm}
\theoremstyle{definition}
\newtheorem{definition}{Definition}[section]
\newtheorem{thm}{Theorem}[section]

\newtheorem{proposition}{Proposition}[section]

\newtheorem{assumption}{Assumption}
\newtheorem{lemma}{Lemma}[section]
\usepackage{amsfonts}

\title{\LARGE \bf
Clustering and synchronization analysis of Networks of Bistable Systems
}

\author{Gianluca Villani and Luca Scardovi% <-this % stops a space
  \thanks{This research was supported by the National Sciences and Engineering Research Council of Canada.}% <-this % stops a space
  \thanks{The authors are with the Department of Electrical and Computer Engineering, University of Toronto, 
    10 King's College Road, Toronto, ON, M5S 3G4, Canada.  
    {\tt\small gianluca.villani@mail.utoronto.ca,  luca.scardovi@utoronto.ca}}%
}

\usepackage[record,abbreviations]{glossaries-extra}
\usepackage{siunitx}

\begin{document}
\maketitle

\begin{abstract}
  This paper studies the dynamics of a network of diffusively-coupled bistable systems. Under mild conditions and without requiring smoothness of the vector field, we analyze the network dynamics and show that the solutions converge globally to the set of equilibria for generic monotone (but not necessarily strictly monotone) regulatory functions. Sufficient conditions for global state synchronization are provided. Finally, by adopting a piecewise linear approximation of the vector field, we determine the existence, location and stability of the equilibria as function of the coupling gain. The theoretical results are illustrated with numerical simulations.
\end{abstract}

\section{Introduction}

Nonlinear dynamical systems can exhibit the coexistence of multiple stable equilibria or, more generally, attractors. This phenomenon, known as multistability \cite{attneave1971multistability}, is ubiquitous across various natural sciences, including optics, mechanics, chemistry, and biology \cite{feudel2008complex}, \cite{pisarchik2014control}. Multistability plays a significant role in fundamental biological processes. In the context of Gene Regulatory Networks (GRNs), genetic circuits that control cell fate decisions are traditionally modeled as multistable dynamical systems \cite{waddington2014strategy}, with different stable steady states representing different cell phenotypes \cite{shah2020reprogramming}. It has been conjectured \cite{thomas1981relation} and proved \cite{snoussi1998necessary} that the presence of positive feedback loops is a necessary condition for the existence of multiple equilibria. The synthetic toggle switch is one of the first examples where these theoretical results have been experimentally verified in a biological system \cite{gardner2000construction}.

%For multistable systems, the properties of stability and convergence to equilibria has been extensively studied in the literature. Existing approaches include monotone systems theory \cite{di2010limit}, Lyapunov functions \cite{chua1988cellular, lin1999complete, smillie1984competitive, pasquini2019convergence, forti2001new}, contraction theory \cite{muldowney1990compound, kcontraction}, or leveraging passivity-like properties \cite{angeli2006systems, angeli2007multistability, angeli2009convergence, ghallab2018extending, ghallab2022negative}.

Recently, networks of bistable switches coupled by quorum sensing \cite{nikolaev2016quorum} and diffusive coupling \cite{ali2019multi, augier2022qualitative, chaves2019coupling, villani2021global} have been investigated, with their properties of auto-correction, synchronization and convergence to equilibria characterized. In \cite{augier2022qualitative}, \cite{chaves2019coupling} the property of convergence to equilibria has been conjectured without a formal analysis. In \cite{nikolaev2016quorum}, the property of (almost) global convergence to equilibria was proved using the theory of strongly monotone systems \cite{smith2008monotone}. However, this property is lost when the graph associated with the signed Jacobian of the vector field is not irreducible everywhere. This occurs, for example, when the interactions between different chemical species are modeled by piece-wise affine functions with intervals where the functions are constant. In \cite{villani2021global} this property has been  established for a particular class of networks with strictly increasing nonlinearities. Alternative approaches to prove convergence to equilibria approaches include  the use of Lyapunov functions \cite{chua1988cellular, lin1999complete, smillie1984competitive, pasquini2019convergence, forti2001new}, contraction theory \cite{muldowney1990compound, kcontraction}, and leveraging passivity-like properties \cite{angeli2006systems, angeli2007multistability, angeli2009convergence, ghallab2018extending, ghallab2022negative}.

In this paper, we investigate a fairly general class of networks of bistable systems coupled by diffusion, focusing on convergence to equilibria and synchronization properties. 
% We introduce a general compartmental model for a network of bistable systems, where each compartment is a positive feedback loop of two chemical species, and extend it to a network of N identical compartments coupled via diffusive coupling. 
The model encompasses both positive autoregulation loops and mutual repression circuits (toggle switches). Our first result establishes that all trajectories converge to the set of equilibria, regardless of the system’s parameters, coupling strength, and network topology, and for generic monotone (not necessarily strictly monotone) regulatory functions. This result builds upon the theory of counterclockwise Input/Output systems \cite{angeli2006systems,  angeli2007multistability, angeli2009convergence} and generalizes results requiring strictly monotone regulation functions \cite{nikolaev2016quorum, villani2021global}. 
% with the introduction of some more general results. 

Our second result provides sufficient conditions for asymptotic state synchronization. This result does not require a specific network topology but establishes a condition that involves the degree of connectivity of the network topology and the isolated system's dynamics. 

Our last result addresses a biologically relevant networked model by adopting a piece-wise linear approximation of the regulatory functions. In particular, we characterize the local stability of equilibria and, for an all-to-all homogeneous network, we analytically characterize the emergence of stable clustering configurations when the strength of the diffusive coupling is varied. 
% The provided analytical bounds can always be computed, independently of the number of compartments considered. 
To the best of our knowledge, such an exhaustive study of convergence, synchronization, and clustering properties has not been previously reported in the literature.

The paper is organized as follows. In Section II, we introduce the class of compartmental models under investigation. Section III establishes the property of convergence to equilibria. In Section IV, we provide sufficient conditions for global synchronization. In Section V, we characterize the local stability of equilibria and clustering properties for an all-to-all network of piecewise affine bistable systems. Finally, in Section VI, we illustrate the results with numerical simulations and suggest directions for future research. The appendix reports the main technical preliminaries used throughout the paper.

\section{A class of Networks of Bistable Systems}
    \subsection{A class of Bistable Motifs}
    In this section, we present a class of interconnected compartmental models that exemplifies a network of bistable systems. 
    %Each compartment consists of two systems interconnected in positive feedback 
    The dynamics of each compartment is described by the ordinary differential equation
    \begin{equation}
      \begin{aligned}
          & \dot{x}_{1}=-\gamma_1 x_{1}+V_1 g_2(x_{2})  \\ %:= f_1(x_1, x_2), \\
          & \dot{x}_{2}=-\gamma_2 x_{2}+V_2 g_1(x_{1}), %:=  f_2(x_1, x_2),
      \end{aligned}
      \label{eq:bistable_compartment}
    \end{equation}
    where the positive scalar values $x_1$ and $x_2$ represent the concentration of two chemical species and the positive constants $V_1$ and $V_2$ determine their synthesis rates. The regulation functions $g_i: \mathbb{R}_{\geq 0} \to \mathbb{R}_{\geq 0}$ are assumed to be monotone functions, positive-valued and either both monotonically increasing or monotonically decreasing\footnote{In this paper we will use the terminology monotonically increasing (decreasing) to indicate that a function is non-decreasing (non-increasing).}. The class of models in \eqref{eq:bistable_compartment} provides a unified framework for studying the dynamics of two important bistable motifs in gene regulatory networks.  
    When both regulation functions are monotonically non-decreasing, \eqref{eq:bistable_compartment} includes Positive Auto Regulation (PAR) loops \cite{alon2006introduction}, where two transcription factors activate each other. Autoinduction circuits, such as those involved in quorum sensing systems \cite{rai2012prediction, fujimoto2013design, schuster2023parameters}, are typically modeled as PAR loops. Conversely, when the regulation functions are both monotonically decreasing (inhibitory interactions), \eqref{eq:bistable_compartment} includes the well-known toggle switch where two proteins mutually inhibit each other's synthesis \cite{gardner2000construction}. Under appropriate assumptions on the regulation functions and the system's parameters, both types of systems exhibit bistability, with two stable steady states. Figure \ref{fig:phase_portrait} shows an example of a phase portrait of the vector field in \eqref{eq:bistable_compartment} where both the regulation functions are piecewise affine functions \eqref{eq:g_function}. 
    
    One of the most common modelling choices for the regulation functions for an activator is the Hill function
    \begin{equation}
      g_s(x) = \frac{(x/\theta)^n}{1 + (x/\theta)^n},
      \label{eq:g_function_smooth} 
    \end{equation}
    where $\theta>0$ defines the concentration of $x$ to significantly activate expression and $n$ is the \textcolor{black}{cooperative degree} corresponding to different steepness values of the function $g_s$. The regulation function for a repressor is instead modelled as
    \begin{equation}
    g^-_s(x) = \frac{1}{1 + (x/\theta)^n} = 1- g_s(x),
    \label{eq:g_function_smooth_repressor} \quad 
    \end{equation}
    where $\theta>0$ represents the concentration of $x$ to significantly repress expression. 
    %The repressor Hill function has the opposite effect compared to the activator Hill function, with the expression level decreasing as the concentration of the repressor increases. 
    The sigmoidal Hill regulatory functions can be approximated with 
     piecewise affine continuous functions \cite{alon2006introduction}, \cite{pwa_gene_modelling_comparison}
    \begin{equation}
      g(x)= \begin{cases}
        0,                 & \ x< \theta                   \\\
    (x-\theta)/\delta, & \theta \leq x \leq \theta+\delta \\
        1                  & \ x > \theta +\delta
      \end{cases}
      \label{eq:g_function}
    \end{equation}
    % or discontinuous step functions
    % \begin{equation}
    %   g_d (x)= \begin{cases}
    %     0, & \ x\leq \theta_d \\
    %     1, & \ x > \theta_d
    %   \end{cases},
    %   \label{eq:g_function_disc}
    % \end{equation}
     with corresponding inhibitory regulatory functions $g^-(x) = 1-g(x)$. 
    % and $g^-_d(x) = 1-g_d(x)$, respectively. 
    In addition to the activator and repressor functions described above, another important case to consider is when one of the functions is a linear increasing function, i.e., $g(x) = x$. This case corresponds to a situation where one of the chemical species is a synthase, and the corresponding synthesized molecule is synthesized at a rate proportional to its concentration \cite{fujimoto2013design}, \cite{schuster2023parameters}. In this paper, unless specified, we will not assume specific regulatory functions but will rely on the following assumption.    
    
    %Based on the discussed regulatory functions, the following assumption will be used through the rest of the paper.

    %This linear function can be seen as a special case of the piecewise affine continuous function in \eqref{eq:g_function}.
    
    %The previously mentioned approximations lead to a series of qualitative models, including Boolean Networks and Piecewise Affine (PWA) models. In this paper, we will focus on PWA models, which have been widely used to study gene regulatory networks and have been shown to capture the essential dynamics of these systems while remaining tractable for analysis **NEED CITATION**. 
    %Therefore, in the rest of the paper, we will restrict to \eqref{eq:g_function} and its inhibitory counterpart under the following assumption

\begin{assumption}\ref{ass:monotonicity_boundedness_lipschitz}
The functions $g_i,\ i=1,2$, in \eqref{eq:bistable_compartment} are positive-valued, Lipschitz continuous, and either both monotonically increasing or both monotonically decreasing. Furthermore, at least one of the functions is bounded.
\label{ass:monotonicity_boundedness_lipschitz}
\end{assumption}
    This assumption guarantees the boundedness of all forward trajectories of \eqref{eq:bistable_compartment}, as shown in the following proposition.
    
    \begin{proposition}
      If Assumption \ref{ass:monotonicity_boundedness_lipschitz} holds, there exist $\bar{x}_1$ and $\bar{x}_2$ such that the set
      \begin{equation}
        \mathcal{B}:=[0, \bar{x}_1]\times [0, \bar{x}_2]
        \label{eq:B_set}
      \end{equation}
      is forward invariant with respect to the dynamics \eqref{eq:bistable_compartment}.
      \label{prop:box_invariance}
    \end{proposition}
    
    \begin{proof}
      %\textcolor{extensions}{ SUBSTITUTE THE NOTATION ABOVE EVERYWHERE}
      To prove that the set $\mathcal{B}$ is invariant, it is sufficient to prove that on the boundary $\partial \mathcal{B}$, the vector field in \eqref{eq:bistable_compartment} satisfies the Nagumo condition \cite{blanchini2008set}. This requires the following conditions to be met on $\partial \mathcal{B}$:
      \begin{itemize}
        \item [i)]  $f_1(0, x_2)\geq 0$, $\forall x_2 \in [0, \bar x_2]$;
        \item  [ii)] $f_2(x_1, 0)\geq 0$, $\forall x_1 \in [0, \bar x_1]$;
        \item  [iii)] $ f_1(\bar x_1, x_2)\leq  0$, $\forall x_2 \in  [0, \bar x_2]$;
        \item  [iv)] $f_2(x_1, \bar x_2)\leq 0$, $\forall x_1 \in [0, \bar x_1]$.
      \end{itemize}
      Since $g_1$ and $g_2$ are nonnegative functions, the first two conditions are trivially satisfied. The last two conditions are equivalent to
      \begin{equation}
        \begin{split}
          \gamma_{1} \bar x_1 \geq V_1g_2(x_2), \forall x_2 \in [0, \bar x_2],\\
          \gamma_{2} \bar x_2 \geq V_2g_1(x_1), \forall x_1 \in [0, \bar x_1].\\
        \end{split}
      \end{equation}
      Without loss of generality, assume that $g_1$ is upper bounded with bound $M>0$. From the second equation, we obtain $\bar x_2 \geq V_2 M/ \gamma_2$. Using this bound in the first inequality we obtain the second bound $\bar x_1 \geq g_2(V_2/\gamma_2)V_1/\gamma_1$. We conclude that $\mathcal{B}$ is forward invariant for any $\bar x_1 \geq  g_2(MV_2/\gamma_2)V_1/\gamma_1$ and $\bar x_2 \geq M V_2/\gamma_2$.\end{proof}

    \begin{figure}
      \centering
      \includegraphics[width=0.99\linewidth]{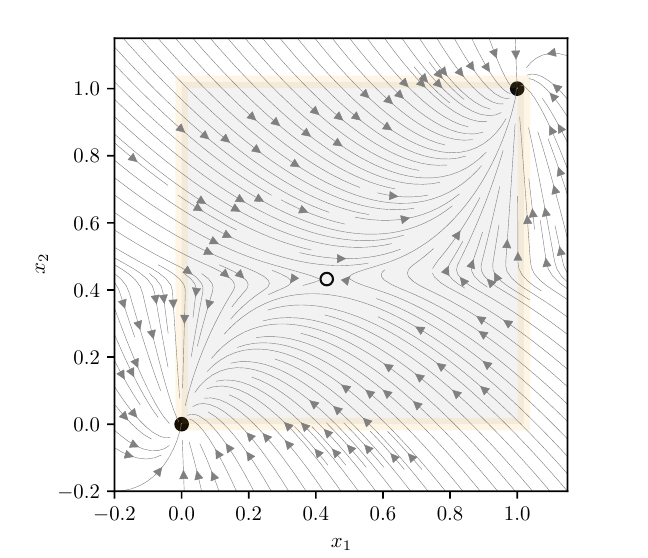}
      \caption{Phase portrait for \eqref{eq:bistable_compartment} with regulatory functions $g$ in \eqref{eq:g_function}. The grey box with orange edges highlights the forward invariant (and attractive) set $\mathcal{B}$. The full black circles represent the two stable equilibria and the white circle represents the unstable saddle equilibrium.}
      \label{fig:phase_portrait}
    \end{figure}
    \subsection{Interconnected bistable motifs}
    Simple network motifs as the one described in \eqref{eq:bistable_compartment} are often not isolated but interact through a common shared medium (\textit{quorum sensing}, \cite{miller2001quorum}, \cite{nikolaev2016quorum}, \cite{boo2021quorum}) or diffusive coupling \cite{silver1996diffusible}, \cite{ali2019multi}. We consider here a networked system with $N$ identical compartments, each defined by \eqref{eq:bistable_compartment}, and coupled via diffusive coupling
    \begin{equation}
      \begin{aligned}
          & \dot{x}_{1, i}=-\gamma_1 x_{1, i}+V_1 g_2(x_{2,i})+\sum_{j} a_{i, j}\left(x_{1, j}-x_{1, i}\right), \\
          & \dot{x}_{2, i}=-\gamma_2 x_{2, i}+V_2g_1(x_{1,i}),
      \end{aligned}
      \label{eq:bistable_network_first}
    \end{equation}
    $i=1,\ldots,N$. The non-negative coefficients $a_{i,j}$ are coupling coefficients between the first species in different compartments. In the rest of the paper we will assume that the communication between compartments relies solely on the diffusion of the first species over an undirected graph, i.e. $a_{i,j} = a_{j,i}$. This assumption is realistic since it often applies that only one of the regulatory molecules can diffuse through the cell's membrane \cite{boo2021quorum, fujimoto2013design, schuster2023parameters }. By defining $X_1 = [x_{1,1}, \dots, x_{1,N}]^T$, $X_2 = [x_{2,1}, \dots, x_{2,N}]^T$ and the vectors of activations $G_\ell(X_\ell)= [g_\ell(x_{\ell,1}), \dots, g_\ell(x_{\ell,N})],\ \ell=1,2$, we can rewrite \eqref{eq:bistable_network_first} in the compact form
    \begin{equation}
      \begin{aligned}
          & \dot{X}_{1}=-(\Gamma_{1} + L) X_{1}+V_{1} G_2(X_{2}), \\
          & \dot{X}_{2}=-\Gamma_{2} X_{2}+V_{2} G_1(X_{1}),
      \end{aligned}
      \label{eq:bistable_network_vector}
    \end{equation}
    where $L=[l_{ij}]$,  is the Laplacian matrix defined as
    \begin{equation*}
      l_{ij} = \left\{\begin{array}{ll}
        -a_{ij}^\ell,                       & i \neq j \\
        \sum\limits_{i \neq j} a_{ij}^\ell, & i=j
      \end{array}, \right. 
    \end{equation*}
 $\Gamma_i = \gamma_i I_N$ with $I_N$ being the $N\times N$ identity matrix. The parameters $a_{ij}$ can be associated to a graph where species in different compartments are the nodes and $a_{ij}>0$ corresponds to a weighted edge in the graph. We will denote the resulting graph as $\mathcal{G}$. The graph $\mathcal{G}$ is connected if given any two nodes there exists a path that connects them and $\mathcal{G}$ is symmetric (or undirected) if $a_{ij} = a_{ji}$ for every $i,j$.  
    %The connectivity properties of the graph $G^\ell$ can be related to the algebraic properties of the Laplacian matrices $L_\ell$. 
    In the following we will use the concept of algebraic connectivity extended to directed graphs \cite{wu2005algebraic}.
      \begin{definition}{\cite{wu2005algebraic}}
        For a directed graph with Laplacian matrix $L$, the algebraic connectivity is the real number defined as
        $$
          \lambda:=\min _{z \in \mathcal{P}} z^T L z
        $$
        where $\mathcal{P}=\left\{z \in \mathbb{R}^n: z \perp \mathds{1}_n,\|z\|=1\right\}$ and where $\mathds{1}_n \triangleq$ $[1,1, \ldots, 1]^T \in \mathbb{R}^n$
      \end{definition}
      %}

    \noindent From the definition of $\mathcal{B}$ in \eqref{eq:B_set} we can define a forward invariant hyperbox for the networked system in \eqref{eq:bistable_network_first}
    \begin{equation}
      \mathcal{B}^N=\prod_{i=1}^{N} \mathcal{B} .
    \end{equation}
    The proof of the invariance of $\mathcal{B}^N$ follows the same lines of the proof for Proposition \ref{prop:box_invariance} and is therefore omitted.
    For the rest of the paper, we will consider $\mathcal{B}^N$ as the state space.

\section{Convergence Analysis}
\label{sec:multistability_clustering}

In the next result, we show that all trajectories of \eqref{eq:bistable_network_first} converge to the set of equilibria, independently of the system's parameters, coupling strength, and network topology. 

\begin{thm}
  Consider system \eqref{eq:bistable_network_first} under Assumption \ref{ass:monotonicity_boundedness_lipschitz} and assume that $a_{i,j}=a_{j,i}$, $i,j=1,\ldots,N$. Then, for every initial condition $\xi \in \mathcal{B}^N$ the $\omega$-limit set $\omega(\xi)$, contains only equilibria of \eqref{eq:bistable_network_first}. If the equilibria of \eqref{eq:bistable_network_first} are isolated, for all initial conditions, $\lim _{t \rightarrow+\infty} x(t)=\bar{x}$, where $\bar{x}$ is an equilibrium of \eqref{eq:bistable_network_first}.
  \label{thm:multistability_convergence}
\end{thm}
In the proof of Theorem \ref{thm:multistability_convergence}, we will make use of the notion of Counterclockwise (CCW) input-output dynamics introduced in \cite{angeli2006systems, angeli2007multistability, angeli2009convergence}. The main definitions and results used in the paper are reported in the Appendix.

\subsection{Proof of Theorem \ref{thm:multistability_convergence}}
As a first step, we rewrite \eqref{eq:bistable_network_first} as the positive feedback interconnection
\begin{equation}
  \begin{split}
    &\dot{x}_{1, i}=-\gamma_1 x_{1, i}+V_1 g_2(x_{2,i}) +u_i \\
    & \dot{x}_{2, i}=-\gamma_2 x_{2, i}+V_2g_1(x_{1,i}) \\
    & y_i = x_{1, i}\\
    & u = -Ly.
  \end{split}
  \label{eq:single_system_network}
\end{equation}
$i=1,\ldots,N$, $y = [x_{1,1}, \dots, x_{1,N}]^T$ and $u=[u_{1}, \dots, u_{N}]$. 
We outline the proof in the following steps. In Proposition \ref{prop:positive_loop} we show that each component of \eqref{eq:single_system_network} has strictly counterclockwise (CCW) input-output dynamics from $u_i$ to $y_i$. Then, by exploiting the properties of positive feedback loop interconnections of systems with CCW input-output dynamics, we characterize the omega limit sets of the networked system \eqref{eq:bistable_network_first}. 

The next Lemma %\ref{prop:scalar_monotone_preliminaries} 
is instrumental in the proof of Proposition \ref{prop:positive_loop} and is an adaptation of \cite[Lemma III.1]{angeli2007multistability}. Note that, unlike the original result in \cite[Lemma III.1]{angeli2007multistability}, we do not require $g$ to be a strictly increasing function of the input.
\begin{lemma}  \label{prop:scalar_monotone_preliminaries}
Consider the scalar input-output system 
\begin{equation}\label{eq:input_with_sat}
\dot{x} = - k x + g(u) = f(x, u),\quad y=h(x)
\end{equation}
where $k>0$, $u \in \mathcal{U}$, and $\mathcal{U}$ is the set of Lebesgue measurable, locally essentially bounded functions valued in $U=[0, u_{\operatorname{max}}]$. Let $X=[x_{\operatorname{min}}, x_{\operatorname{max}}], x_{\operatorname{max}}\geq g(u_{\operatorname{max}})/k$ and $x_{\operatorname{min}} \leq g(0)/k$, be the state space of \eqref{eq:input_with_sat}. Assume that $h$ and $g$ are either both monotonically increasing or both monotonically decreasing Lipschitz continuous functions. Then, \eqref{eq:input_with_sat} has CCW input-output dynamics with respect to arbitrary density functions $\rho$, i.e. 
\begin{equation}\label{eq:ccw_inequality}
    \liminf\limits_{T\to +\infty} \int\limits_{0}^{T} \dot{y}(t)\int\limits_{0}^{u(t)}\rho(\eta, h(x(t))) d\eta\ dt > -\infty 
\end{equation}
for every $x_0 \in X$, $u \in \mathcal{U}$, and $\rho$ satisfying Definition   \ref{def:density_function}.
\end{lemma}
\begin{proof}
Let $g$ and $h$ be monotonically increasing functions. For all input values $u(t) \in U$, $f(x_{\operatorname{min}}, u(t)) \geq 0$ and $f(x_{\operatorname{max}}, u(t)) \leq 0$, therefore the state space $X$ is positively invariant under the dynamics defined by \eqref{eq:input_with_sat}. We claim that the set $(g(0), g(u_{\operatorname{max}}))$ is also positively invariant. For any initial condition $x_0\in (g(0)/k, g(u_{\operatorname{max}})/k)$, the corresponding solution is given by
    \begin{equation}\label{eq:solution_input_sat}
    \begin{aligned}
          x(t) = &\operatorname{e}^{-k t}x_0 + \int\limits_{0}^t \operatorname{e}^{-k (t-\tau)} g(u(\tau)) d\tau.
    \end{aligned}
    \end{equation}
    Since $ g(0)\leq g(u(t))\leq g(u_{\operatorname{max}})$, from \eqref{eq:solution_input_sat} we obtain
    \begin{equation}
    x(t)\geq g(0)/k+\operatorname{e}^{-k t}\left(x_0 -g(0)/k\right) > g(0)/k
    \end{equation}
    and 
    \begin{equation}
        x(t) \leq  g(u_{\operatorname{max}})/k+\operatorname{e}^{-k t}\left(x_0 -g(u_{\operatorname{max}})/k\right) < g(u_{\operatorname{max}})/k,
    \end{equation}
thus proving the claim.
 \noindent We can therefore define the partition $X \times \mathcal{U} = \mathcal{S}_{\mathrm{l}} \cup \mathcal{S}_{\mathrm{m}} \cup \mathcal{S}_{\mathrm{h}}$, where
 \begin{align*}
    \mathcal{S}_{\mathrm{l}} &= \left\{(x_0, u): x(t, x_0, u(t))\in [x_{\operatorname{min}}, g(0)/k],\forall t\geq 0\right\}, \\
        \mathcal{S}_{\mathrm{h}}&= \left\{(x_0, u) : x(t, x_0, u(t)) \in [g(u_{\operatorname{max}})/k, x_{\operatorname{max}}], \forall t\geq 0\right\},\\
\mathcal{S}_{\mathrm{m}}&= \left\{(x_0, u) : \exists t^*\geq0:  \forall t\geq t^*,
    \right.\\\left.
    \right.& \left. x(t, x_0, u(t))\in (g(0)/k, g(u_{\operatorname{max}}/k)\right.\}.
    \end{align*}
We now show that \eqref{eq:ccw_inequality} holds on each component of the partition.
Consider a generic pair $x_0 \times u \in \mathcal{S}_{\mathrm{l}}$. By definition of $\mathcal{S}_{\mathrm{l}}$ and since $g$ is non-decreasing,   $f(x(t), u(t)) =  - k x(t) + g(u(t))\geq 0$, for every $t\geq 0$.  
Since $h$ is Lipschitz continuous, its derivative is defined almost everywhere and, by assumption, $Dh(x):=\frac{\operatorname{d}}{\operatorname{d}x}h(x) \geq 0$, where defined. We conclude that $\dot{y}(t) = D h \left(x(t)\right) f(x(t), u(t)) \geq 0, \forall t\geq 0$ almost everywhere and therefore
\begin{equation} \label{eq:ccw_property}
   \liminf \limits_{T \to + \infty} \int_{0}^{T} \dot{y}(t)\int_{0}^{u(t)} \rho(\mu, y(t))d\mu dt > -\infty,
\end{equation}
 \\
for every $\rho$ satisfying Definition A.1.

We now show that the inequality \eqref{eq:ccw_inequality} holds for all pairs $x_0 \times u \in \mathcal{S}_{\mathrm{m}}$. Since the function $g$ is non-decreasing, $\forall x \in (g(0)/k, g(u_{\operatorname{max}})/k)$ we can define  $\gamma(x) := g^{-1}(kx)$ where $g^{-1}$ is the generalized inverse function of the monotone function $g$, in the sense of Definition \ref{def:gen_inverse_modified}. The function $\gamma(x)$ is strictly increasing by Proposition \ref{prop:pseudo_inverse_properties}, and bounded.
Since $\dot{y} =  D h(x) f(x, u)$, $\dot{y}>0 \implies f(x, u)=-k x+g(u)>0 \iff g(u)>k x \implies u > \gamma(x)$ with the final implication following from Proposition \ref{prop:pseudo_inverse_properties}. Analogously, $\dot{y}<0 \implies u < \gamma(x)$ and therefore we obtain the following inequality  
\begin{equation} \label{eq:S_m}
     \dot{y}(t) \int\limits_{0}^{u(t)}\rho(\eta, y(t)) d\eta \geq \dot{y}(t) \int\limits_{0}^{\gamma(x(t))}\rho(\eta, y(t)) d\eta.
\end{equation}
The function
\begin{equation}
    F(x) := \int \limits_{0}^x  D h(\xi) \int\limits_{0}^{\gamma(x)}\rho(\eta, h(\xi))d\eta  d\xi
\end{equation}
is Lipschitz in $x$ since it is the integral of a bounded function. Furthermore, $F(x(t))$ is absolutely continuous, since it is the composition of a Lipschitz function with the absolutely continuous function $x(t)$. We conclude that its time derivative 
\begin{align}
    \frac{\operatorname{d}}{\operatorname{d}t}F(x(t)) 
    = \dot{y}(t)\int\limits_{0}^{\gamma(x(t))}\rho(\eta, y(t)) d\eta 
\end{align}
 is defined almost everywhere.
Integrating both sides of \eqref{eq:S_m} we obtain
     \begin{equation*}
         \int \limits_{0}^T  \dot{y}(t) \int\limits_{0}^{u(t)}\rho(\eta, y(t)) d\eta dt \geq F(x(T)) - F(x(t^*)) + \Delta
     \end{equation*}
where 
\begin{equation*}
    \Delta = \int \limits_{0}^{t^*}  \dot{y}(t) \int\limits_{0}^{u(t)}\rho(\eta, y(t)) d\eta dt
\end{equation*}
is finite and accounts for the interval of time during which the solution $x(t)$ is outside of $(g(0)/k, g(u_{\operatorname{max}})/k)$. Therefore, by the continuity of $F$ and boundedness of the trajectories, \eqref{eq:ccw_inequality} holds. Consider now a generic pair $x_0 \times u \in \mathcal{S}_{\mathrm{h}}$. From \eqref{eq:input_with_sat} it holds that  $\dot{y}(t) = Dh (x(t)) f(x(t), u(t)) \leq 0, \forall t \geq 0$ almost everywhere, and therefore
\begin{align*}
     \dot{y}(t)\int_{0}^{u(t)} \rho(\mu, y(t))d\mu dt \geq 
     \dot{y}(t)\int_{0}^{u_{\operatorname{max}}} \rho(\mu, y(t))d\mu dt .
    %&=\\
    %\liminf \limits_{T \to + \infty} K \left(y(T) - %y(0)\right)>-\infty.&
\end{align*}
%where $K=\int_{0}^{u_{\operatorname{max}}} \rho(\mu, y(t))d\mu>0$.
Integrating both sides we obtain
\begin{equation}
         \int \limits_{0}^T  \dot{y}(t) \int\limits_{0}^{u(t)}\rho(\eta, y(t)) d\eta dt \geq P(x(T)) - P(x(0))
\end{equation}
where
\begin{equation}
    P(x) = \int \limits_{0}^x  D h(\xi) \int\limits_{0}^{u_{\operatorname{max}}}\rho(\eta, h(\xi))d\eta  d\xi
\end{equation}
from which \eqref{eq:ccw_inequality} follows.

The proof for the case of $g$ and $h$ mononotically decreasing follows the same lines. 
\end{proof}

\begin{proposition}
  \label{prop:positive_loop}
  Consider the following two-dimensional system 
  \begin{equation}
    \begin{split}
      & \dot{x}_{1} = -\gamma_1 x_{1} + V_1 g_2(x_{2}) + u\\
      & \dot{x}_{2} = -\gamma_2 x_{2} + V_2 g_1(x_{1})\\
      & y = x_{1}
    \end{split}
    \label{eq:positive_loop_first}
  \end{equation}
  where $g_i$, $i=1,2$, satisfy Assumption \ref{ass:monotonicity_boundedness_lipschitz}. Assume that the input signal $u(t) \in [ u_{\operatorname{min}}, u_{\operatorname{max}}]$ is bounded. Then, the system has strictly CCW input-output dynamics from $u$ to $y$.
\end{proposition}
\begin{proof}
System \eqref{eq:positive_loop_first} can be represented as a positive loop interconnection of the $x_1$ and $x_2$ subsystems 
  \begin{equation}
    \begin{split}
      \dot{x}_1 & = -\gamma_1 x_1 + v + u = f_1(x_1, v, u)  \\
      \dot{x}_2 & = -\gamma_2 x_2 +  V_2g_1(x_{1}) \\
      v &= V_1 g_2(x_2) \\
      y &= x_1.
    \end{split}
  \end{equation}
Assumption \ref{ass:monotonicity_boundedness_lipschitz} ensures bounded states and output for bounded input signals with $u(t) \in [u_{\operatorname{min}}, u_{\operatorname{max}}]$. Proposition \ref{prop:scalar_monotone_preliminaries} guarantees that the $x_2$ subsystem with input $x_1$ and output $v$ has CCW input-output dynamics with respect to arbitrary density functions. Furthermore, $f_1$ is $\mathcal{C}^1$, strictly increasing with respect to $u$ and $v$ and $\partial f_1/ \partial u=1>0$. Therefore, by Lemma III.2 in \cite{angeli2007multistability} we can conclude that the system in \eqref{eq:positive_loop_first} has strictly CCW input-output dynamics with respect to arbitrary density functions from $u$ to $y$.
\end{proof}

\begin{figure}
  \centering
   \includegraphics[width=0.92\linewidth]{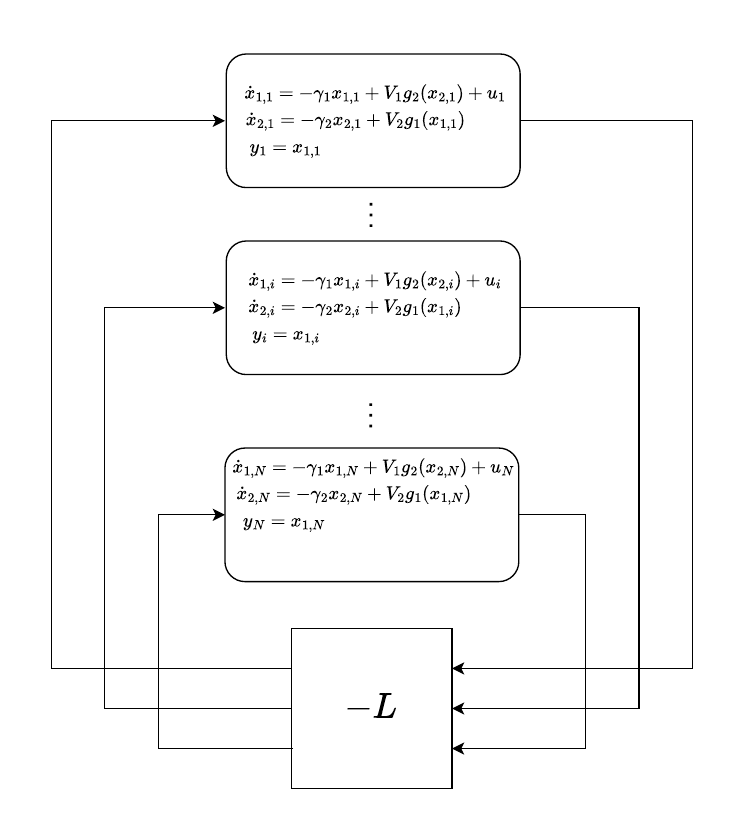}
  \caption{Positive feedback interconnection scheme representing the mathematical model in \eqref{eq:single_system_network}. }
  \label{fig:cocoercivitydiagram}
\end{figure}

By leveraging the results above, we can now conclude the proof of Theorem \ref{thm:multistability_convergence}. System \eqref{eq:single_system_network} is a bank of SISO scalar nonlinear systems in feedback with the static map $-L: y \to -Ly$ where $L$ is the symmetric Laplacian matrix in \eqref{eq:bistable_network_vector}. In Proposition \ref{prop:positive_loop}, we showed that each of the subsystems in \eqref{eq:single_system_network} with input $u_i$ and output $y_i$ has strict CCW input-output dynamics. Furthermore, the map $-L: y \to -Ly$ is a  system with CCW input-output dynamics since the Laplacian matrix $L$ is symmetric \cite[Proposition 6.2]{angeli2009convergence}. Finally, by \cite[Theorem 1]{angeli2009convergence} we can conclude that for any initial condition $\xi$, the $\omega$-limit set $\omega(\xi)$ is contained in the largest invariant set in $M:=\{x:\ \dot{y} \equiv 0\}$ where  $y = [x_{1,1}, \dots, x_{1,N}]^T$.  Furthermore, as the forward orbits are bounded, the omega limit set $\omega(\xi)$ is non-empty, compact and connected \cite[Proposition 1.111]{chicone2006ordinary}. In $M$, since $\dot{y}=0$, each component ${x}_{1,i}$ must be constant, and we write ${x}_{1,i}(t)=\bar{x}_{1,i}$. 
Therefore, the dynamics of each compartment is independent with respect to each other and reads
\begin{equation}
  \begin{split}
    & \dot{x}_{1,i} = - \gamma_1 \bar{x}_{1,i} + V_1 g_2(x_{2,i}) + \overline{b}_i = 0\\
    & \dot{x}_{2,i} = - \gamma_2 x_{2,i} + V_2 g_1(\bar{x}_{1,i}) = - \gamma_2 x_{2,i} + c_i,
  \end{split}
  \label{eq:omega_ccw}
\end{equation}
where $\overline{b}_i=\sum_{j} a^1_{i, j}\left(\overline{x}_{1, j}-\overline{x}_{1, i}\right)$ is the constant term depending on the diffusive coupling and $c_i= V_2g_1(\bar{x}_{1,i})$.  Consider an initial condition $\xi_0$ in $\omega(\xi)$ and assume that $\xi_0$ is not an equilibrium. The invariance of $\omega(\xi)$ implies that the solution $x(t)$ corresponding to the initial condition $\xi_0$ is entirely contained in $\omega(\xi)$, that is, $x(t)\in \omega(\xi), \forall t\in \mathbb{R}$. Furthermore, for the dynamics in \eqref{eq:omega_ccw}, 
$\lim\limits_{t \to \pm \infty} |x_{2,i}(t)|=+\infty$ against the fact that $\omega(\xi)$ is bounded. Therefore, all points in $\omega(\xi)$ must be equilibria.
Furthermore,  if the equilibria of the system are isolated, then for all initial conditions $ \xi \in \mathcal{B}^N$, $\lim\limits_{t\to + \infty}\norm{x(t) - \overline{x}} \to 0$ where $\overline{x}$ is an equilibrium of \eqref{eq:bistable_network_first}.

\section{Synchronization Analysis}
In the following, we derive sufficient conditions on the algebraic connectivity of the coupling graph that ensure that all the equilibria of \eqref{eq:bistable_network_vector} are synchronized. 
\begin{thm} \label{prop:global_synchronization_bound}
    Consider a generic equilibrium $\overline{X} = [\overline{X}_1, \overline{X}_2]$ of  \eqref{eq:bistable_network_vector}. Assume that Assumption 1 holds and let  $\ell_1$ and $\ell_2$ be the Lipschitz constants of $g_1$ and $g_2$. Let $\lambda_2$ be the algebraic connectivity of the Laplacian matrix $L$. If the following condition is satisfied:
    \begin{equation}\label{eq:synchronization_condition}
        \frac{V_1\ell_1 V_2 \ell_2}{(\gamma_1+\lambda_2)\gamma_2} <1,
    \end{equation}
  then $\overline{X}_1 \in \operatorname{span}\{{\mathds{1}_N}\}$ and $\overline{X}_2 \in \operatorname{span}\{{\mathds{1}_N}\}$, where $\mathds{1}_N \in \mathbb{R}^N$ is a vector with all entries equal to one.
\end{thm}

\begin{proof}
    For a generic equilibrium  $\overline{X} = [\overline{X}_1, \overline{X}_2]$ of  \eqref{eq:bistable_network_vector}, it must hold that
\begin{equation}
    (\Gamma_{1} + L)\overline{X}_{1}=\tilde{G}(\overline{X}_1)
\end{equation}
where $\tilde{G}(\overline{X}_1):=G_2(\frac{V_2}{\gamma_2}G_1(\overline{X_1}))$.
If such an equilibrium exists, then  
\begin{equation}
    Q^T Q(\Gamma_{1} + L)\overline{X}_{1}= Q^T Q \tilde{G}(\overline{X}_1),
\end{equation}
where
\begin{equation*}
    Q=\left[\begin{array}{ccccc}
-1+(N-1) \nu & 1-\nu & -\nu & \cdots & -\nu \\
-1+(N-1) \nu & -\nu & 1-\nu & \ddots & \vdots \\
\vdots & \vdots & \ddots & \ddots & -\nu \\
-1+(N-1) \nu & -\nu & \cdots & -\nu & 1-\nu
\end{array}\right],
\end{equation*}
and
\begin{equation*}
    \nu=\frac{N-\sqrt{N}}{N(N-1)} \ .
\end{equation*}
From the definition of $Q$, it follows that  $Q \mathds{1}_N=0$, and $Q^T Q  = I_N-\frac{1}{N} \mathds{1}_N \mathds{1}_N^T$. Therefore, we can write 
\begin{equation*}
    Q^T Q(\Gamma_{1} + L)\left(\overline{X}_{1} -\alpha \mathds{1}_N\right) = Q^T Q\left(\tilde{G}(\overline{X}_{1}) - \tilde{G}(\alpha \mathds{1}_N) \right),
\end{equation*}
where $\alpha = \frac{1}{N}\mathds{1}_N^T \overline{X}_{1}$. Assume that $\overline{X}_1 \notin \operatorname{span}\{\mathds{1}_N\}$. Then
\begin{equation} \label{eq:l_Q_reduced}
\norm{Q^T Q(\Gamma_{1} + L)\left(\overline{X}_{1} -\alpha \mathds{1}_N\right)}_2 \geq (\gamma_1 + \lambda_2) \norm{\overline{X}_{1} -\alpha \mathds{1}_N}_2
\end{equation}
where $\lambda_2$ is the algebraic connectivity of the Laplacian $L$. $\tilde{G}$ has a Lipschitz constant equal to $\frac{V_1\ell_1 V_2 \ell_2}{\gamma_2}$, and therefore
\begin{equation} \label{eq:lipschitz_inequality}
\norm{ Q^T Q\left(\tilde{G}(\overline{X}_{1}) - \tilde{G}(\alpha \mathds{1}_N) \right)}_2 \leq  \frac{V_1\ell_1 V_2 \ell_2}{\gamma_2} \norm{\overline{X}_{1} -\alpha \mathds{1}_N}_2 .
\end{equation}
Since
\begin{equation*} \label{}
\frac{V_1\ell_1 V_2 \ell_2}{(\gamma_1+\lambda_2)\gamma_2} <1,
\end{equation*}
the inequalities in \eqref{eq:l_Q_reduced} and \eqref{eq:lipschitz_inequality} lead to a contradiction, as it must hold that:
\begin{align*}
&\norm{Q^T Q(\Gamma_{1} + L)\left(\overline{X}_{1} -\alpha \mathds{1}_N\right)}_2 = \\
&\norm{Q^T Q\left(\tilde{G}(\overline{X}{1}) - \tilde{G}(\alpha \mathds{1}_N) \right)}_2 . 
\end{align*}
Therefore, for all equilibria $\overline{X} = [\overline{X}_1, \overline{X}_2]$, $\overline{X}_1$ $ \in \operatorname{span}\{\mathds{1}_N\}$. Since $\overline{X}_2 = V_2/\gamma_2 G_1(\overline{X_1})$, we conclude  that $\overline{X}_2$ $ \in \operatorname{span}\{\mathds{1}_N\}$.
\end{proof}
\begin{remark}
In virtue of Theorem \ref{prop:global_synchronization_bound} and Theorem \ref{thm:multistability_convergence}, we can conclude that if the algebraic connectivity of the coupling graph is sufficiently large, all the solutions of \eqref{eq:bistable_network_vector} asymptotically converge to synchronized equilibria, where the states of each compartment approach each other to a constant value.         
\end{remark}

% \begin{remark}
% The result in Theorem \ref{prop:global_synchronization_bound} can be proven          
% \end{remark}

\section{Multistability Analysis}
\iftrue
  \label{section:all_explicit_sync}

  This section focuses on the multistability analysis of a subclass of the network models \eqref{eq:bistable_network_first}. 
  %network of bistable systems modeled by a Piecewise Affine vector field. 
  Specifically, let $g_2(x) = x$ and $g_1(x) = g(x)$, a piecewise linear function defined in \eqref{eq:g_function} and characterized by the positive parameters $\theta$ and $\delta$. 
  The resulting system reads
   \begin{equation}
      \begin{aligned}
          & \dot{x}_{1, i}=-\gamma_1 x_{1, i}+V_1 x_2+\sum_{j} a_{i, j}\left(x_{1, j}-x_{1, i}\right), \\
          & \dot{x}_{2, i}=-\gamma_2 x_{2, i}+V_2g(x_1).
      \end{aligned}
      \label{eq:bistable_network_cluster}
    \end{equation}
The dynamics of this piecewise affine model can be seen as an approximation of the dynamics of $N$ bistable switches interconnected by a quorum sensing mechanism. Indeed, in the case of an all-to-all interconnection, it is analogous to those described in \cite{rai2012prediction, fujimoto2013design, schuster2023parameters}, in the limit that autoinducer diffusion is much more rapid than degradation. Under this assumption, intracellular and extracellular concentrations of the autoinducer molecule can be approximated to be equal, as verified in experimental studies, in different biological systems \cite{kaplan1985diffusion}.

In the second part of this section, for an all-to-all homogeneous interconnection, we characterize the local stability of the equilibria of \eqref{eq:bistable_network_cluster} and investigate their location as the system's parameters and coupling strength are varied.

  The following assumption will be used throughout the remainder of this section.
  \begin{assumption}
    The system parameters of \eqref{eq:bistable_network_cluster} satisfy
    $$
      \frac{V_1 V_2}{\delta \gamma_1 \gamma_2}>1+\frac{\theta}{\delta}.
     $$
    \label{ass:bistability_bounds}
  \end{assumption}
  \vspace{-17pt}
  Assumption \ref{ass:bistability_bounds} guarantees that the uncoupled system has two stable equilibria
  $P_{\mathrm{ON}}=\left(\left(V_1 V_2\right) /\left(\gamma_2 \gamma_1\right), V_2 / \gamma_2\right), P_{\mathrm{OFF}}=(0,0)$ and an unstable equilibrium $P_{\mathrm{s}}=\left(\bar{x}_1, \bar{x}_2\right)$ with
  $$
    \begin{aligned}
    \bar{x}_1 &  = \frac{V_1}{\gamma_1} \bar{x}_2\\
      \bar{x}_2 & =\frac{V_2 \theta}{\delta \gamma_2}\left(\frac{V_1 V_2}{\delta \gamma_1 \gamma_2}-1\right)^{-1}. 
    \end{aligned}
  $$

  It is useful to partition the phase space $\mathcal{B}^N$ into different \textit{domains} such that the restriction of the vector field is linear-affine in each of them.
  \begin{definition} Let  $\mathcal{T} \subset \mathbb{R}^{N}$ denote the set of $N$ dimensional vectors with entries in $\{-1, 0, +1\}$. We can associate each element $\alpha$ in $\mathcal{T}$ with a \textit{domain} $\Omega_{\alpha}$ defined as
    \begin{equation*}
      \begin{aligned}
        \Omega_{\alpha}:=\left\{x\in\right. \mathcal{B}^N \mid & x_{1, j} > \theta + \delta,\text{ if } \alpha_{j}=1                 \\
               & x_{1, j} < \theta,\text{ if } \alpha_{j}=-1    \\
        \theta\leq   & x_{1, j}\leq\theta + \delta,\ \left.\text{if }\alpha_{j}=0\right\}.
      \end{aligned}
      \label{eq:def_omega}
    \end{equation*}
    A domain $\Omega_{\alpha}$ is called \textit{saturated} if $\alpha \in \Lambda_{e} :=\{\alpha \in \mathcal{T}: \alpha_{j} \in \{-1, 1\}\}$, \textit{linear} if $\alpha \in \Lambda_0 :=\{\alpha \in \mathcal{T}: \alpha_{j} = 0\}$, and \textit{mixed} if $\alpha \in \Lambda_m := \mathcal{T}/\left(\Lambda_0 \cup \Lambda_e \right)$.
  \end{definition}
    %\vspace{-18pt}
  %\textcolor{red}
  {\begin{definition}
      Given a saturated domain $\Omega_{\alpha}$, we say that $\Omega_{\alpha}$ has an average level of activation $\overline{q}:=m/N,\ m=0, \dots,\ N$ where $m$ is the number of entries in $\alpha$ equal to $1$. 
    \end{definition}}
    
  In the following, we characterize the stability of equilibria in the different types of domains. Similarly to \eqref{eq:bistable_network_vector}, we define the state vector $x=[X_1^T, X_2^T]^T$ such that 
 the restriction of the vector field in \eqref{eq:bistable_network_cluster} to a domain $\Omega_\alpha$ is described by the linear-affine system
    \begin{equation}
      \dot{x} = M_{\alpha} x + b_{\alpha},
    \end{equation}
    where
    \begin{equation}
      M_\alpha = \begin{bmatrix}
        M_{11}(\alpha) & M_{12}(\alpha) \\
        M_{21}(\alpha) & M_{22}(\alpha)
      \end{bmatrix}
      \label{eq:M_partition}
    \end{equation}
    with
    \begin{equation}\begin{split}
        M_{11}(\alpha) &= -(\Gamma_1 + L)\\
        M_{12}(\alpha) &= V_1 I_N\\
        M_{21}(\alpha) &= V_2/\delta\  \operatorname{diag}(w(\alpha_1), \dots, w(\alpha_N))\\
        M_{22}(\alpha) &= -\Gamma_2
      \end{split}
      \label{eq:M_description}
    \end{equation}
    and $w(\alpha_{i})=1$ if $\alpha_i=0$ and $w(\alpha_{i})=0$, otherwise.
    We skip here the explicit expression of $b_\alpha$ since it is not involved in the subsequent calculations. The following proposition characterizes the local stability of the equilibria of \eqref{eq:bistable_network_cluster}.

  \begin{proposition}
  Consider system \eqref{eq:bistable_network_cluster}.
    If an equilibrium exists in a saturated domain, it is unique and locally asymptotically stable. Furthermore, if  $    \frac{V_1V_2}{\gamma_1\gamma_2\delta} > N$, the equilibria in the interior of non-saturated domains are unstable.
    \label{prop:instability_sufficient_conditions}
  \end{proposition}
  \begin{proof} % Correct M_21
      For all saturated domains, $M_{21}(\alpha)=0$, therefore the matrix $M_\alpha$ is block diagonal, nonsingular and Hurwitz. Since the vector field is differentiable in each saturated domain, if an equilibrium exists in a saturated domain $\Omega_{\alpha}$, it is unique and locally asymptotically stable. We now prove that the equilibria in the interior of mixed and linear domains are unstable if  $\frac{V_1V_2}{\gamma_1\gamma_2\delta} > N$. Let $\Omega_{\alpha}$ be a mixed or linear domain, we claim that the matrix $M_\alpha$ has at least one eigenvalue with strictly positive real part. We proceed by contradiction. Assume that $M_{\alpha}$ has no eigenvalues with strictly positive real part. Then, \begin{equation}
      M_\alpha^\varepsilon := \begin{bmatrix}
        M_{11}(\alpha) -\varepsilon I_N & M_{12}(\alpha) \\
        M_{21}(\alpha) & M_{22}(\alpha) -\varepsilon I_N
      \end{bmatrix}
      \label{eq:M_partition}
    \end{equation}
    must be Hurwitz, $\forall \varepsilon>0$. Since $M_{\alpha}^\varepsilon$
      is Metzler, it is Hurwitz if and only if both $M_{22}(\alpha) - \varepsilon I_N$ and
      the Schur complement $ M_\alpha^\varepsilon / \left(M_{22}(\alpha) - \varepsilon I_N\right)$ are Hurwitz \cite[Corollary 1]{souza2017note}. Since  $M_{22}(\alpha) - \varepsilon I_N$  is clearly Hurwitz, $M_\alpha^\varepsilon$ is Hurwitz  if and only if 
    \begin{equation*}
      \begin{split}
        M_\alpha^\varepsilon / \left(M_{22}(\alpha) - \varepsilon I_N\right) =  M_{11}(\alpha) -\varepsilon I_N + \frac{V_1 }{(\gamma_2+\varepsilon)}M_{21}(\alpha)\\
      \end{split}
    \end{equation*}
    is Hurwitz.
    Notice that $  M_\alpha^\varepsilon / \left(M_{22}(\alpha) - \varepsilon I_N\right)$ is symmetric and
    \begin{equation*}
      \mathds{1}_N^T  M_\alpha^\varepsilon / \left(M_{22}(\alpha) - \varepsilon I_N\right) \mathds{1}_N = -N(\gamma_1+\varepsilon) \gamma_1 + \frac{V_2 V_1 }{(\gamma_2+\varepsilon)\delta} n_\alpha %>0
      \label{eq:mixed_instability}
    \end{equation*}
    with $\mathds{1}_N$ the vector whose entries are all equal to 1 and $n_\alpha \in\{1,..., N\}$ depends on the domain $\Omega_{\alpha}$ and represents the number of species in the linear regime.
    By assumption, $\frac{V_1V_2}{\gamma_1\gamma_2\delta} > N$,  therefore there exists (small enough) $\varepsilon>0$ such that $\forall n_\alpha \in\{1,..., N\}$
    \[
     \mathds{1}_N^T  M_\alpha^\varepsilon / \left(M_{22}(\alpha) - \varepsilon I_N\right)  \mathds{1}_N>0.
    \]
Therefore, for small enough $\varepsilon>0$  
$M_{\alpha}^\varepsilon$ is not Hurwitz, thus proving the claim. 

  \end{proof}

  In the following, we study the location of (locally asymptotically stable) equilibria in the saturated domains as a function of the coupling strength $k$ for the case of a homogeneous all-to-all network, i.e. $L =  k\left(N I_N-\mathds{1}_N\mathds{1}_N^{\operatorname{T}}\right)$.

  Notice that in the two saturated domains with average level of activation $\overline{q}=0$ and $\overline{q}=1$, the equilibria exist and are unique, each corresponding to all compartments synchronized in the state $P_{\operatorname{OFF}}$ and $P_{\operatorname{ON}}$. In the rest of the paper, the following definition will be used to refer to a specific class of equilibria:
\begin{definition}
    A generic equilibrium    $\overline{X} = [\overline{X}_1, \overline{X}_2]$ of  \eqref{eq:bistable_network_vector} is said to be synchronized if $\overline{X}_1 \in \operatorname{span}\{{\mathds{1}_N}\}$ and $\overline{X}_2 \in \operatorname{span}\{{\mathds{1}_N}\}$.
\end{definition}

  \begin{thm}
    \label{thm:synchr_saturated}
    Consider system \eqref{eq:bistable_network_cluster} and assume $L =  k\left(N I_N-\mathds{1}_N\mathds{1}_N^{\operatorname{T}}\right)$.
    Let $\Omega_{\alpha}$ be a saturated domain with an average level of activation $\overline{q}=m/N$, $m \in \{1, \dots, N-1\}$. There exists a minimum gain $k^{\overline{q}}$ such that $\forall k > k^{\overline{q}}$, no equilibria are contained in the domain $\Omega_{\alpha}$.
    Additionally, there exists $k^s>0$ such that for every $k>k^s$, if an equilibrium exists in a saturated domain, it is synchronized.
  \end{thm}  
  \begin{proof}
    % $\overline{x} = [\overline{x}_1, \overline{x}_2] \in \Omega_{\alpha}$ of  
    A generic equilibrium  $\overline{X} = [\overline{X}_1, \overline{X}_2]$ of  \eqref{eq:bistable_network_vector} must satisfy
    \begin{equation}
      \begin{aligned}
          & \overline{X}_{1}= V_1 (\Gamma_{1} + L)^{-1} \overline{X}_2 \\
          & \overline{X}_{2}=V_{2}/\gamma_2 G_1(\overline{X}_{1}).
      \end{aligned}
      \label{eq:bistable_network_vector_equilibrium}
    \end{equation}
    As $\Gamma_1 + L = k\left(N I_N-\mathds{1}\mathds{1}^{\operatorname{T}}\right)+\gamma_{1} I_N$, its inverse (Sherman-Morrison formula \cite{sherman1950adjustment, max1950inverting}) is
    \begin{equation}
      (\Gamma_1+ L)^{-1}=\frac{1}{N k+\gamma_{1}} I_N+\frac{1}{\gamma_{1}\left(N k+\gamma_{1}\right)} k \mathds{1}_N\mathds{1}_N^{\operatorname{T}}
    \end{equation}
    and by substituting it in \eqref{eq:bistable_network_vector_equilibrium}, we obtain
    \begin{equation}
      \overline{X}_{1}=V_{1}\left(\frac{1}{N k+\gamma_{1}} I_N+\frac{1}{\gamma_{1}\left(N k+\gamma_{1}\right)} k \mathds{1}_N\mathds{1}_N^{\operatorname{T}} \right) \frac{V_2}{\gamma_2} G_1(\overline{X}_{1}).
    \end{equation}
    Therefore, for a generic compartment $i$,
    \begin{equation}
      \begin{aligned}
          & \overline{x}_{1,i} = \frac{V_1 V_2/\gamma_2}{Nk + \gamma_1} g (\overline{x}_{1,i})+ \frac{V_1 V_2/\gamma_2}{\gamma_1(Nk + \gamma_{1})} N k \sum\limits_{i=1}^{N} g(\overline{x}_{1,i})/N \\
          & \overline{x}_{2,i}=V_{2}/\gamma_2 g_1(\overline{x}_{1,i}).
      \end{aligned}
      \label{eq:eq_saturated}
    \end{equation}

    Define $I_{\operatorname{ON}}(\alpha)=\{i: \alpha_i=1\}$ and $I_{\operatorname{OFF}}(\alpha)=\{i: \alpha_i=-1\}$. As the domain $\Omega_{\alpha}$ has an average level of activation $\overline{q}=m/N$, this means that $I_{\operatorname{ON}}(\alpha)$ and $I_{\operatorname{OFF}}(\alpha)$ have cardinality $m$, and $N-m$, respectively. Denote with $V:=
    \frac{V_1 V_2}{\gamma_1 \gamma_2}$,
    the equilibrium \textbf{ $\overline{X} = [\overline{X}_1, \overline{X}_2] \in \Omega_{\alpha}$} if and only if
    \begin{equation}
      \begin{split}
        \overline{x}_{1, i} &= \gamma_1\frac{V}{Nk + \gamma_1} + \frac{V}{Nk + \gamma_{1}} N k \overline{q} > \theta +\delta 
      \end{split},\  i \in I_{\operatorname{ON}}(\alpha)
      \label{eq:first_equation_true}
    \end{equation}
    and
    \begin{equation}
      \begin{split}
        \overline{x}_{1, i} &=  \frac{V}{Nk + \gamma_{1}} N k \overline{q} < \theta 
      \end{split},\ i \in I_{\operatorname{OFF}}(\alpha).
      \label{eq:second_equation_true}
    \end{equation}
    Therefore,   $\overline{X} = [\overline{X}_1, \overline{X}_2] \notin \Omega_{\alpha}$ if at least one of \eqref{eq:first_equation_true} and \eqref{eq:second_equation_true} is not verified, i.e. if
    \begin{equation}
      \gamma_1\frac{V}{Nk + \gamma_1} + \frac{V}{Nk + \gamma_{1}} N k \overline{q} \leq  {\theta} +  {\delta}
      \label{eq:first_equation_og}
    \end{equation}
    or
    \begin{equation}
      \frac{V}{Nk + \gamma_{1}} N k \overline{q} \geq  {\theta}
      \label{eq:second_equation_og}.
    \end{equation}
    We can now rewrite \eqref{eq:first_equation_og},   \eqref{eq:second_equation_og} as
    \begin{equation}
      kN\left(-V \overline{q} + ( {\theta}+ {\delta})  \right)  \geq \gamma_1 \left(V - ( {\theta} + {\delta}) \right)
      \label{eq:first_inequality}
    \end{equation}
    and
    \begin{equation}
      \left(V \overline{q} -  {\theta}\right)Nk \geq  {\theta} \gamma_{1}
      \label{eq:second_inequality}.
    \end{equation}
    The factor on the right handside of \eqref{eq:first_inequality} is always strictly positive under Assumption \ref{ass:bistability_bounds}. We now show that for any given $\overline{q}$, it is possible to find a minimum gain $k^{\overline{q}}$ such that for all $k > k^{\overline{q}}$, no equilibria exist in the domain $\Omega_{\alpha}$ with an average level of activation $\overline{q}$. We divide the analysis according to the value of $\overline{q}$.

    {1)} If $\overline{q}$ is such that $  {\theta}   <V \overline{q} < ( {\theta}+ {\delta})  $, by simple manipulation of \eqref{eq:first_inequality} and \eqref{eq:second_inequality} we obtain:
    \begin{equation}
      k\geq k_1^{\overline{q}} := \frac{\gamma_{1}\left(( {\theta} + {\delta}) - V\right)}{N\left(V \overline{q} - ( {\theta}+ {\delta})  \right)}
    \end{equation}
    or
    \begin{equation}
      k\geq k^{\overline{q}}_2:= \frac{ {\theta} \gamma_{1}}{N\left(V\overline{q} - {\theta} \right)}.
    \end{equation}Therefore,  $\forall k \geq k^{\overline{q}} = \operatorname{min}\{ k^{\overline{q}}_1,   k^{\overline{q}}_2\}$,  there are no equilibria in the saturated domains corresponding to the \textit{average level of activation} $\overline{q}$.
    
    {2)} If  $\overline{q}$ is such that $V \overline{q} \geq \left( {\theta}  + {\delta}\right)$, no positive value of $k$ can satisfy \eqref{eq:first_inequality}. From \eqref{eq:second_inequality} we can instead get the bound
    \begin{equation}
      k\geq k^{\overline{q}}_2:= \frac{ {\theta} \gamma_{1}}{N\left(V\overline{q} - {\theta} \right)}.
    \end{equation}
    Therefore, $\forall k \geq k^{\overline{q}}_2$ there are no equilibria in the saturated domains corresponding to the \textit{average level of activation} $\overline{q}$.
    
    {3)} If  $\overline{q}$ is such that $V \overline{q} \leq  {\theta}  \gamma_{1}$, no positive value of $k$ can make \eqref{eq:second_inequality} true. From \eqref{eq:first_inequality} we can instead get the bound
    \begin{equation}
      k\geq k_1^{\overline{q}} := \frac{\gamma_{1}\left(( {\theta} + {\delta}) - V\right)}{N\left(V \overline{q} - ( {\theta}+ {\delta})  \right)}.
    \end{equation}
    Therefore, $\forall k \geq k^{\overline{q}}_1$ there are no equilibria in the saturated domains corresponding to the \textit{average level of activation} $\overline{q}$. Finally, for all
    \begin{equation}\label{eq:k_s_definition}
        k> k^{\operatorname{s}}:= \max\limits_{\overline{q} \in \{m/N \,|\, m \in \{1, \dots, N-1\}\}} k^{\overline{q}},
    \end{equation} the only saturated domain with equilibria are the two saturated domains corresponding to the average level of activation $\overline{q}=0$ and $\overline{q}=1$.
  \end{proof}

  \section{Numerical Results}

\begin{figure*}[t]    
    \centering
    \begin{subfigure}{.5\textwidth}
      \centering
      \includegraphics[width=1\linewidth]{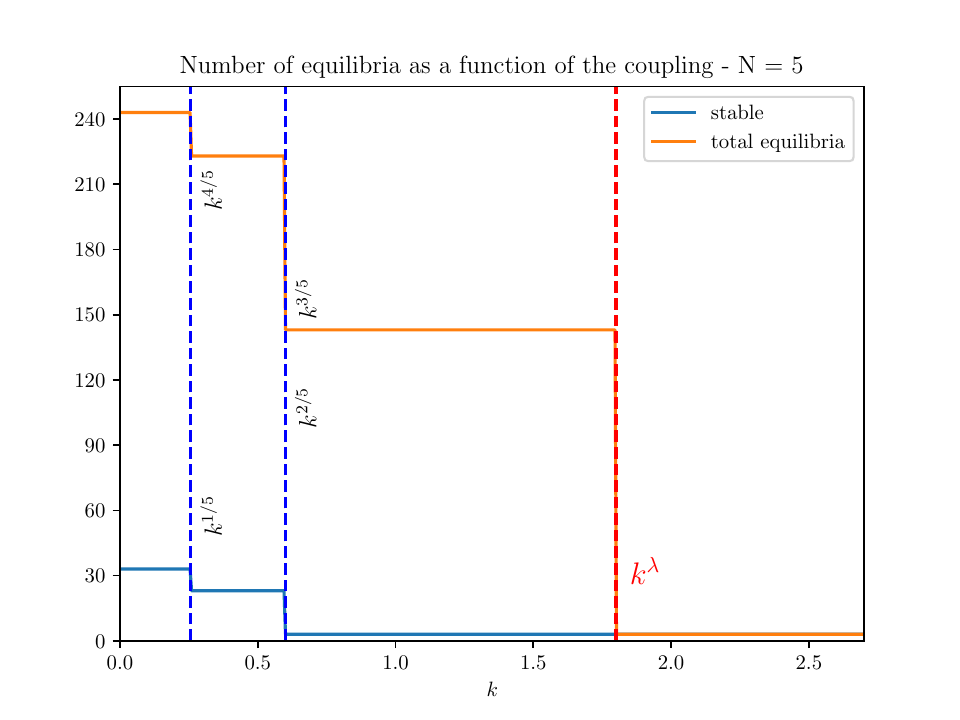}
      \subcaption{All-to-all topology}
      \label{fig:sfig1}
    \end{subfigure}%
    \begin{subfigure}{.5\textwidth}
      \centering
      \includegraphics[width=1\linewidth]{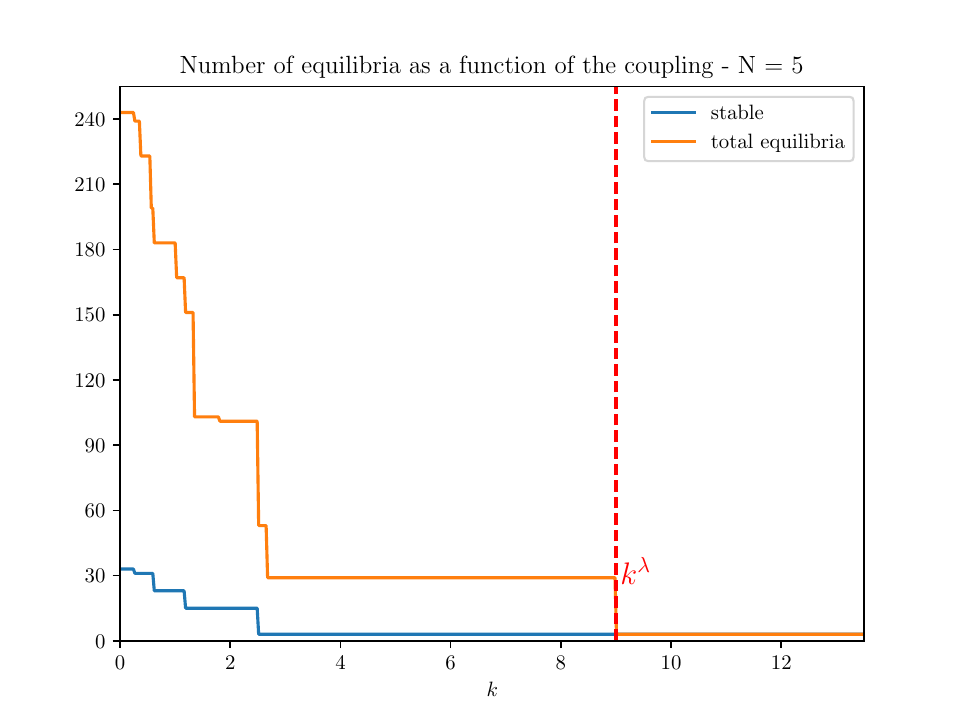}
      \subcaption{Star topology}
      \label{fig:sfig2}
    \end{subfigure} \\
    \centering
    \begin{subfigure}{.5\textwidth}
      \centering
      \includegraphics[width=1\linewidth]{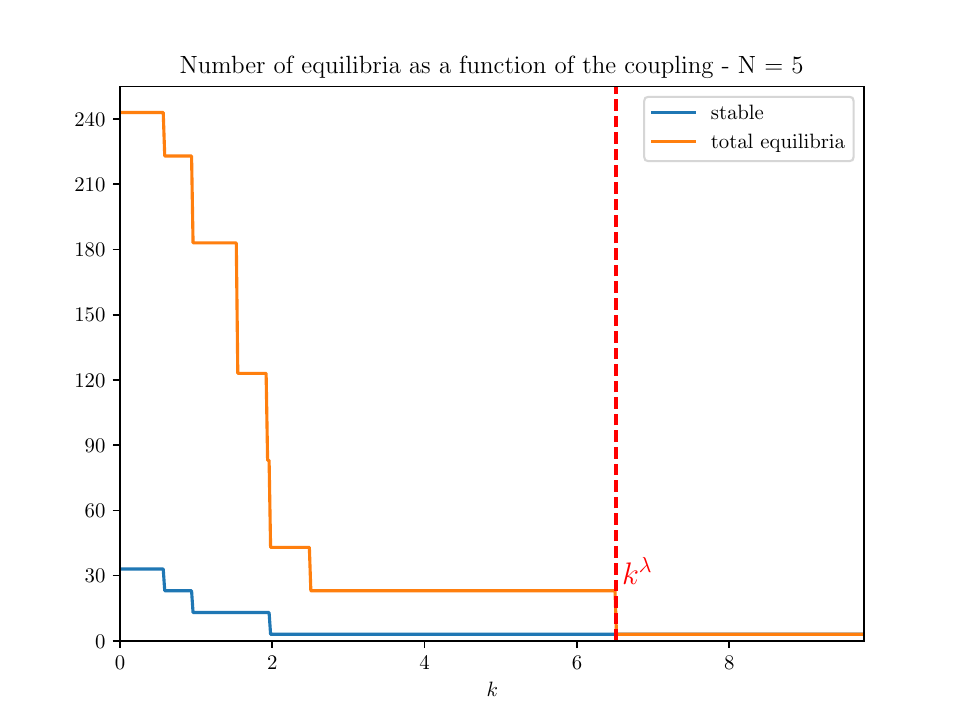}
      \subcaption{Loop topology}
      \label{fig:sfig1}
    \end{subfigure}%
    \begin{subfigure}{.5\textwidth}
      \centering
      \includegraphics[width=1\linewidth]{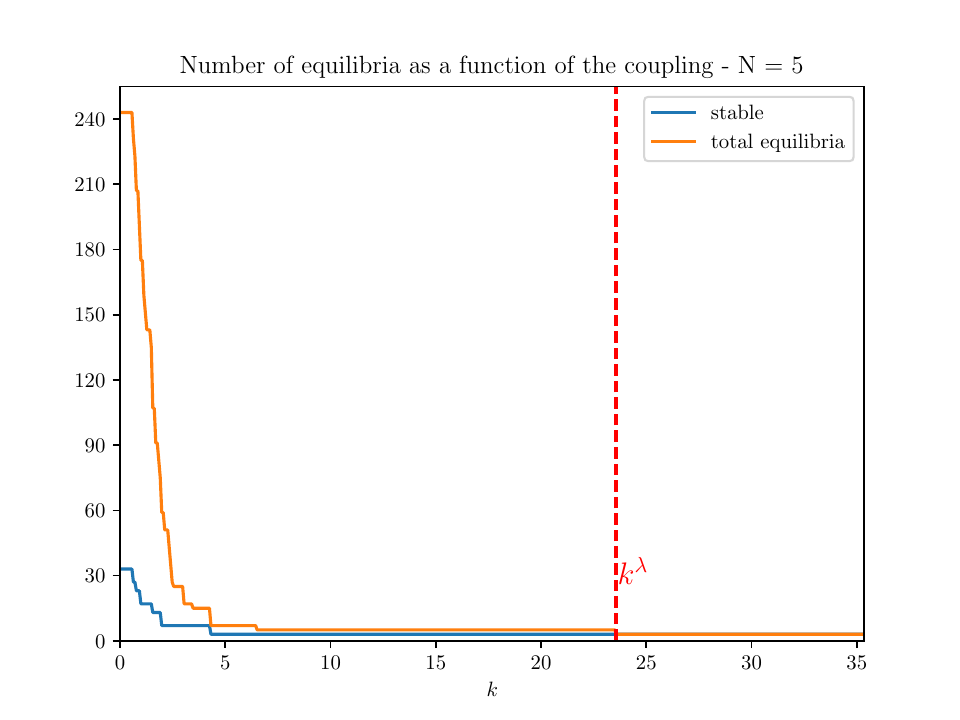}
      \subcaption{Line topology}
      \label{fig:sfig2}
    \end{subfigure}
    \caption{ In (a), the dotted vertical blue lines correspond to the values $k^{\overline{q}}$ computed analytically. The continuous blue line represents the total number of equilibria in all saturated domains as a function of the coupling parameter $k$. The continuous orange line represents the total number of equilibria in all domains (including non-saturated) as a function of the coupling parameter $k$. The dotted red line represents the minimum coupling gain $k^{\lambda}$ to ensure that only synchronized equilibria exist. In all figures, the minimum values for the blue and orange lines are respectively 2 and 3. The system's parameters used are are $V_1 = V_2 = 1$, $\gamma_{1}=\gamma_{2}=1$ and $\theta=0.45$, $\delta=0.1$.}
  \end{figure*}

  To illustrate the results developed in the previous sections, we consider a network \eqref{eq:bistable_network_cluster} of $N=5$ identical bistable compartments with $g_2(x_2) = x_2$, and $g_1(x_1)=g(x_1)$ defined as a piecewise function \eqref{eq:g_function} with parameters $V_1=V_2=1$, $\gamma_{1}=\gamma_{2}=1$, $\theta = 0.45$, $\delta=0.1$. The chosen parameters and activation functions guarantee that all the solutions converge to the set of equilibria (Theorem \ref{thm:multistability_convergence}) and that all equilibria in the interior of non-saturated domains are unstable (Proposition \ref{prop:instability_sufficient_conditions}).
  Figure 3 (a)-(d) illustrates the number of equilibria (both in the saturated and unsaturated domains) as a function of the homogeneous coupling 
$k$ for various interconnection topologies. The minimum coupling $k^{\lambda} =\left( \frac{V_1 V_2}{\gamma_2 \delta} - \gamma_1\right)/N$ (red dotted line), computed using Theorem \ref{prop:global_synchronization_bound}, represents the minimum coupling gain that ensures global synchronization.
  % , i.e. for each initial condition, the trajectories of the system converge to a synchronized equilibrium. 
  This is illustrated in Figure 3, as for $k>k^{\lambda}$, the network has only three synchronized equilibria, corresponding to the three equilibria of the uncoupled bistable system.  
  
  In Figure 3 (a) we depict the exact thresholds $k^{\overline{q}}$ (dotted blue lines) computed by using Theorem \ref{thm:synchr_saturated} in the case where the communication topology is homogeneous and all-to-all. These thresholds represent the coupling strength required to eliminate all the stable equilibria with an average level of  activation $\overline{q}$. In particular, we observe a gradual transition from multistability to bistability of the overall network. Indeed, for all $k>k^{\operatorname{s}}$, all stable equilibria are synchronized equilibria (Theorem \ref{thm:synchr_saturated}). The figure also suggests that the bound $k^{\operatorname{s}}$ guarantees convergence to the synchronized equilibria almost everywhere. It is worth commenting on the difference between $k^{\lambda}$ and $k^{\operatorname{s}}$. The bound $k^{\lambda}$ guarantees that only the three synchronized equilibria exist while $k^{\operatorname{s}}$ guarantees that the only locally asymptotically stable equilibria are the synchronized ones.     
  
  The gap between the bounds $k^{\lambda}$ and $k^{\operatorname{s}}$ provides a possible explanation on why, in many examples (e.g. the synchronization of Goodwin oscillators in \cite{scardovi2010synchronization}), the minimum coupling gain that guarantees synchronization is higher than the values empirically found in simulation. 
  % Indeed, unstable steady states are not 'visible' in simulation as the measure of the set of initial conditions for which solutions converge to these attractors has \textbf{likely?} measure zero.
  
  This observation is particularly insightful when we consider the limiting case of $\delta \to 0$, i.e.  when the piecewise affine continuous function \eqref{eq:g_function} tends to a discontinuous function. 
  % \eqref{eq:g_function_disc}. 
  In this case, the minimum coupling gain that guarantees global synchronization computed with \eqref{eq:synchronization_condition} grows unbounded with the Lipschitz constant $\ell_1=1/\delta$ of the function $g_1$. Nevertheless, the minimum coupling gain computed in Theorem \ref{thm:synchr_saturated} approaches a finite value (Figure \ref{fig:k_s}).  
  
  % One could think that networks of bistable swtiches with steep sigmoidal functions are harder to synchronize in practice as Theorem \ref{prop:global_synchronization_bound} seems to suggest. Indeed,   
  % Therefore, the presence of steep sigmoidal functions only requires a higher gain for the unstable non-synchronized equilibria to disappear, but it doesn't affect significantly the coupling strength required for convergence of trajectories to one of the asymptotically stable synchronized equilibria.

\begin{figure}[h!]
  \centering
  \includegraphics[width=0.95\linewidth]{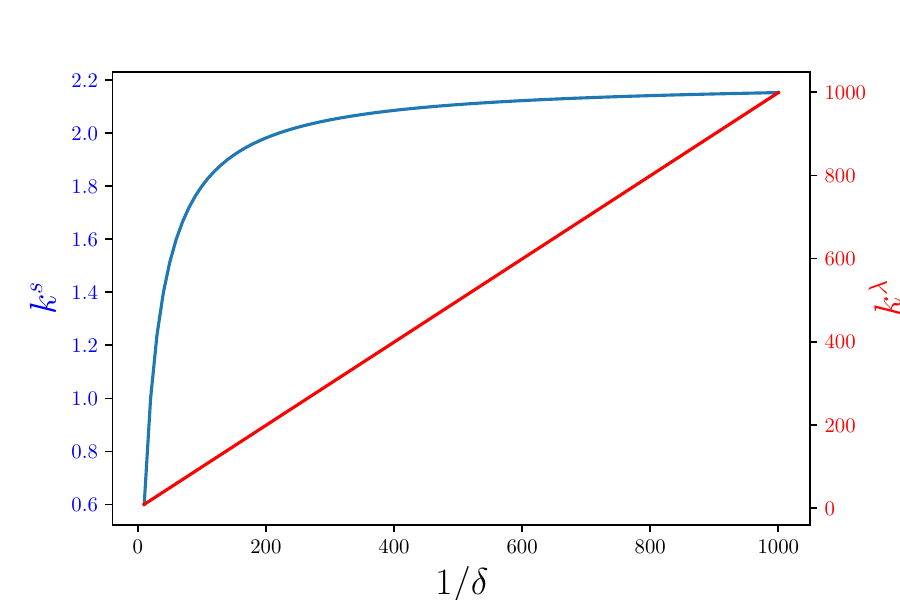}
  \caption{Synchronization bounds $k^{\lambda}$ (red line) and $k^s$ (blue line) for increasing values of the Lipschitz constant $\ell_2=1/\delta$ of the piecewise affine activation function $g_1(x_1) = g(x_1)$. Parameters used: $N=5$, $V_1 = V_2 = 1$, $\gamma_{1}=\gamma_{2}=1$ and $\theta=0.45$.}
  \label{fig:k_s}
\end{figure}

In our final experiments, we compare two different choices for the regulatory functions : piecewise affine approximations \eqref{eq:g_function} and hill-like smooth functions 
\eqref{eq:g_function_smooth}.
We evaluate how closely these two modelling choices agree in the characterization of clustering and synchronization properties of the networked system. According to \cite{pwa_gene_modelling_comparison}, these models should produce similar results when Hill coefficients are sufficiently large, with the exact threshold on the cooperativity degree depending on the specific parameters used. 

\begin{figure}[h]
  \centering
  \includegraphics[width=0.99\linewidth]{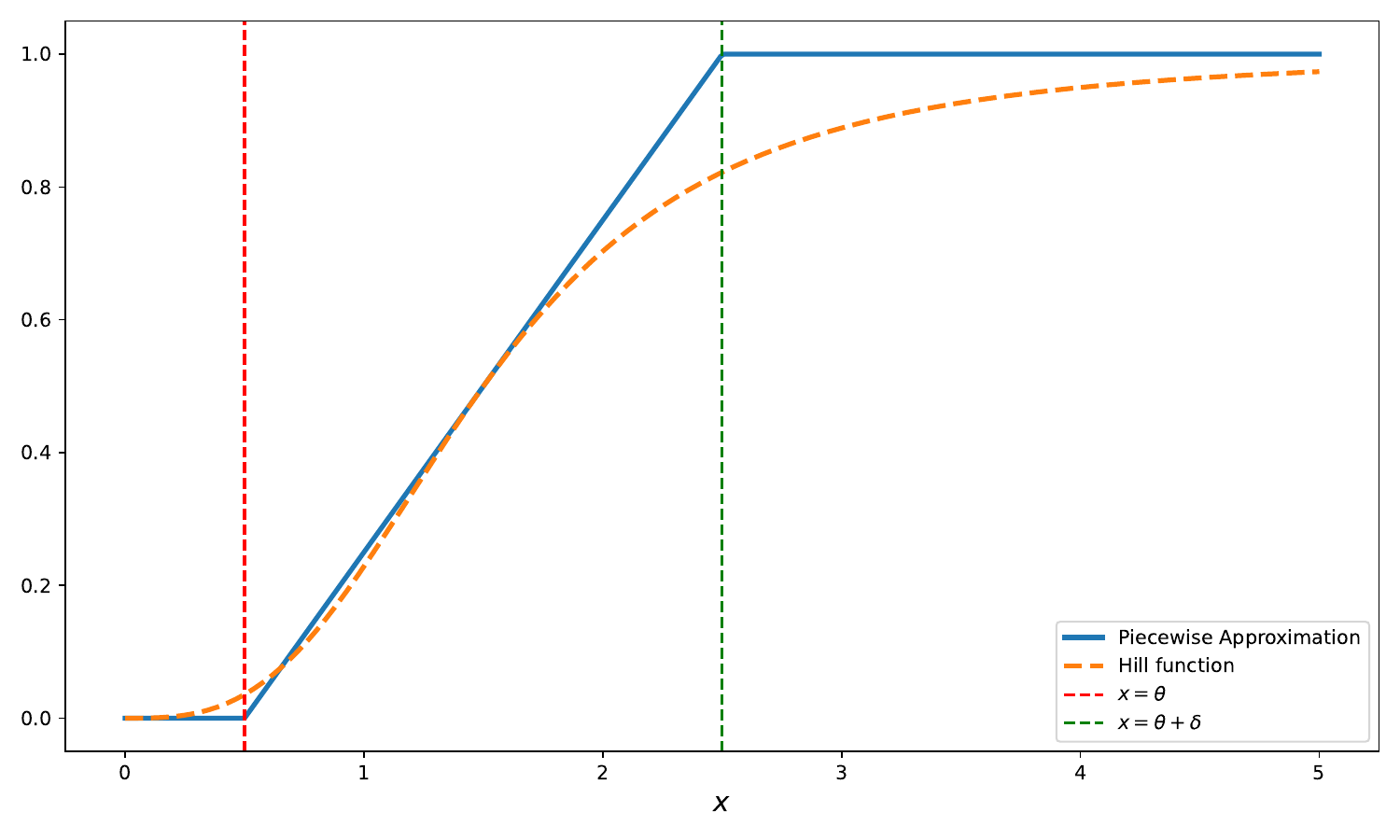}
  \caption{Hill function $g_s (x)$ with cooperativity degree $n=3$ (orange line) and piece-wise affine approximation $g(x)$ (blue line). Parameters used: 
$\theta_H = 1.5$, 
$\delta = 2$, 
$\theta = \theta_H-\delta/2$.}
  \label{fig:comparison_regulation_functions}
\end{figure}
\begin{figure}[h]
  \centering
  \includegraphics[width=0.99\linewidth]{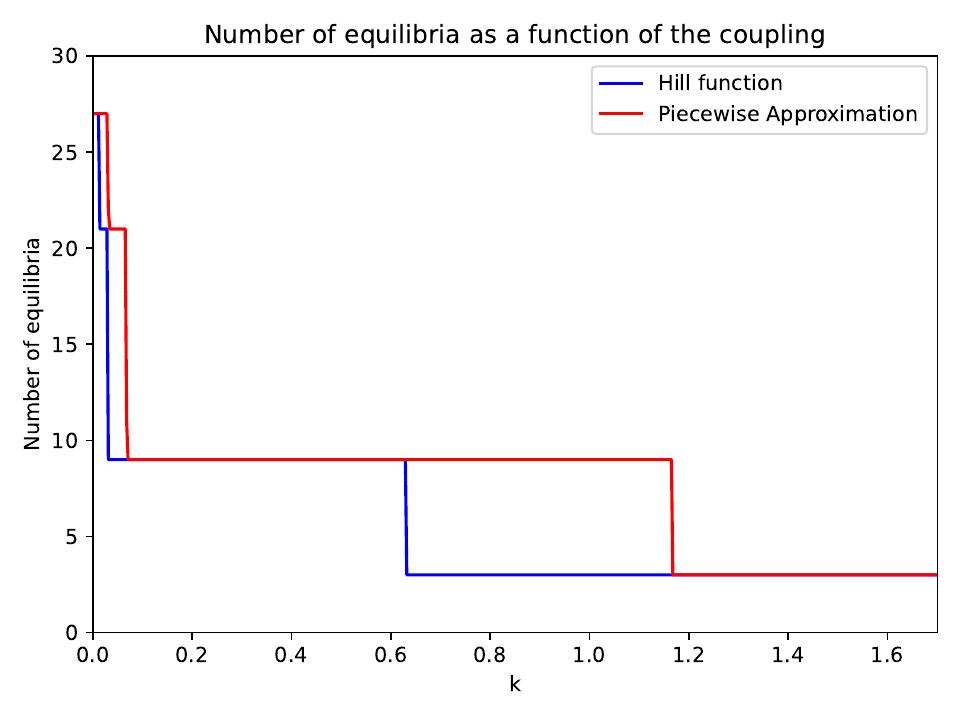}
  \caption{Number of equilibria for different values of $k$ in two variants of a all-to all network of \eqref{eq:bistable_network_first} with $N=3$ bistable compartments with $g_1(x_1) = g_s(x_1)$ (blue line) and $g_1(x_1) = g(x_1)$ (red line). For both lines, the minimum number of equilibria is $3$. For both variants $g_2(x_2) = x_2$. Parameters used: $\theta_H = 1.5$, 
$\delta = 2$, 
$\theta = \theta_H-\delta/2$, $\gamma_1 = \gamma_2 =1$, $V_1 = V_2 = 3$.}
  \label{fig:comparison_n_equlibria}
\end{figure} 

We consider two variants of a network \eqref{eq:bistable_network_first} with $N=3$ identical bistable compartments. In both variants, $g_2(x_2) = x_2$. For $g_1(x_1)$, we compare:

\begin{itemize}
    \item A Hill function \eqref{eq:g_function_smooth} $g_1(x_1) = g_s (x_1)$  with $\theta_H = 1.5$ and cooperative degree $n=3$;
    \item A piecewise affine activation function \eqref{eq:g_function} $g_1(x_1) = g(x_1)$ with parameters $\delta=2$ and $\theta = \theta_H - \delta/2 = 0.5$.
\end{itemize}

The parameters for the two regulatory functions are chosen such that  $g_s(x)|_{x = \theta_H} = g (x)|_{x = \theta_H}$ and  $\frac{\operatorname{d}}{\operatorname{d}x} g_s(x)|_{x = \theta_H} = \frac{\operatorname{d}}{\operatorname{d}x} g(x)|_{x = \theta_H} $ as shown in Figure \ref{fig:comparison_regulation_functions}. The comparison of the number of equilibria as a function of the coupling gain $k$ for the two model variants in Figure \ref{fig:comparison_n_equlibria} indicates that when the Hill function is employed, non-synchronized equilibria are eliminated at slightly smaller values of the coupling gain $k$. The number of equilibria in the Hill function model is more sensitive to changes in the coupling gain $k$.  Despite this difference, both variants of the models exhibit a similar overall behavior. 
This suggests that the methodology developed in this paper provides an alternative to the, computationally heavy, numerical computations to locate the equilibria in the smooth system.

%The experiment shows that for the model using the Hill function, the elimination of equilibria occurs at slightly lower values of the coupling gain $k$. This suggests that mixed equilibria in the Hill function model are more sensitive to changes in $k$.

\fi

\section{Conclusion}

In this paper, we examined the dynamics of a network of bistable systems coupled by diffusion. Under assumptions, we presented structural conditions (i.e., conditions that do not depend on the values of the system’s parameters) that ensure all solutions converge to the set of equilibria. Additionally, under the same technical assumptions, we provided sufficient conditions for the coupling graph to guarantee the bistability of the network and demonstrated that all solutions asymptotically converge to one of the equilibria of the uncoupled system.

For a biologically relevant model of a network of bistable systems, we utilized a piecewise linear approximation of the vector field to determine the location of equilibria as a function of the coupling gain and studied their local stability. This approach provided valuable insights into the impact of diffusive coupling on the stability and synchronization of bistable systems.

Future research directions include extending our results to more general bistable and multistable systems and incorporating other coupling mechanisms such as quorum sensing. Exploring these additional coupling mechanisms could further enhance our understanding of complex network dynamics in biological systems and expand the applicability of our findings to a broader range of real-world scenarios.

\appendix

\subsection{Systems with Counterclockwise Input-Output Dynamics}
   The following is from \cite{angeli2007multistability}.
   Consider the nonlinear differential equation with input $u$ and output $y$
\begin{equation}
\begin{aligned}
\dot{x} &= f(x, u), \ y = h(x),
\end{aligned}
\label{eq:state_space_gen_ccw}
\end{equation}
defined on the closed state space $X \subset \mathbb{R}^n$, where $f: X \times U \to \mathbb{R}^n$ is a locally Lipschitz function, with input signal $u \in \mathcal{U}$ (the set of $U\subset \mathbb{R}^m$-valued Lebesgue measurable locally essentially bounded functions) and the output map $h: X \to Y\subset \mathbb{R}^m$ is locally Lipschitz. The input-output transition map $\psi(t, \xi, u)$ can then be defined, for each initial condition $\xi \in X$ and each input signal $u \in \mathcal{U}$, by
\begin{equation}
\psi(t, \xi, u) = h(x(t, \xi, u)).
\end{equation}
This map relates the input signal $u$ and the initial state $\xi$ to the output of the system at time $t$.
Static input-output maps can be defined as
\begin{equation}
\psi(t, u) = h(u(t)),
\end{equation}
for some Lipschitz function $h: U \to Y$.  The transition map $\psi(t, \xi, u)$ need not be defined for all $t \geq 0$; however, for each $\xi \in X$ and each $u \in \mathcal{U}$ there exists $T_{\xi, u} \in$ $(0,+\infty]$ so that $\psi(t, \xi, u)$ is well defined for all $t \in\left[0, T_{\xi, u}\right)$.
It is useful to define the following subsets of $X \times \mathcal{U}$ :
$$
\begin{aligned}
\mathcal{S}_{\mathrm{fc}} \doteq\left\{(\xi, u) \in X \times \mathcal{U}: T_{\xi, u}=+\infty\right\} \\
\mathcal{S}_{b d} \doteq\left\{(\xi, u) \in X \times \mathcal{U}: T_{\xi, u}=+\infty\right. \\
\quad \text { and } u(\cdot), \psi(\cdot, \xi, u) \text { are bounded }\}
\end{aligned}
$$
% where ' $f c$ ' comes from the commonly adopted term forward complete, used to denote global existence of solutions forward in time.
We assume, without loss of generality that $0 \in U$, and that $U=U_1 \times U_2 \cdots \times U_m$ and $Y=Y_1 \times Y_2 \cdots \times Y_m$ for some nonempty intervals $U_i, Y_i \subset \mathbb{R}$. We are now ready to report some definitions.

\begin{definition} \label{def:density_function}\cite{angeli2007multistability}
We say that $\rho: U \times Y \rightarrow \mathbb{R}^m$ is a density function if it satisfies the following properties:
1) $\rho(u, y)=\left[\rho_1\left(u_1, y_1\right), \rho_2\left(u_2, y_2\right), \ldots, \rho_m\left(u_m, y_m\right)\right]$ for scalar functions $\rho_i: U_i \times Y_i \rightarrow \mathbb{R}$;
2) $\rho_i\left(u_i, y_i\right)>0$ for almost all $\left(u_i, y_i\right) \in U_i \times Y_i$ (according to Lebesgue measure) and all $i$ in $\{1 \ldots m\}$;
3) $\rho$ is a measurable and locally summable function (jointly in $u$ and $y$ ).
\end{definition}

\begin{definition} \cite{angeli2007multistability}
    We say that a system has CCW input-output (CCW input-output) dynamics with respect to the density function $\rho(u, y)$ if for any $(\xi, u) \in \mathcal{S}_{b d}$ the following inequality holds
$$
\liminf _{T \rightarrow+\infty} \int_0^T \dot{y}(t)^{\prime} \int_0^{u(t)} \rho(\mu, y(t)) d \mu d t>-\infty,
$$
where $y(t)=\psi(t, \xi, u)$ is assumed to be absolutely continuous.
\end{definition}

    \begin{definition} \cite{angeli2007multistability}
    	We say that a system has a strict counterclockwise input-output dynamics  with respect to the density function $\rho(u, y): U \times Y \to \mathbb{R}_{\ge 0}$ if the following inequality holds for all pairs $(\xi, u) \in \mathcal{S}_{b d}$
    	
    	\begin{equation*}
    	\liminf \limits_{T \to + \infty} \int_{0}^{T} \hspace{-6pt}\dot{y}(t)'\int_{0}^{u(t)} \hspace{-10pt}\rho(\mu, y(t))d\mu -\frac{\tilde{\rho}(|\dot{y}(t)|)}{1+ \gamma(|x(t)|)} dt> - \infty,
    		\label{eq:ccw_propery}
    	\end{equation*}
     where $\tilde{\rho}$ is a positive definite function, $\gamma \in \mathcal{K}$ and $y(t)=\psi(t, \xi, u)$ is absolutely continuous.
    \end{definition}

    \begin{remark}
    The assumption of $y(t)$ being absolutely continuous is a fundamental one as it ensures that the time derivative $\dot{y}(t)$ is defined almost everywhere. Such a property is always guaranteed under the assumption that the output map $h$ is Lipschitz.
    \end{remark}

\subsection{Generalized inverse}
The following definition of generalized inverse of a monotone function is taken from
\cite{de2015study}.

\begin{definition}\cite{de2015study} \label{def:gen_inverse}
 Let $f$ be a monotonically increasing function. The generalized inverse is function $f^{-1}$ defined by
\begin{equation}
    f^{-1}(y)=\inf \{x \in \overline{\mathbb{R}}: f(x)>y\},\  \overline{\mathbb{R}} =  \mathbb{R} \cup \{-\infty, +\infty\}
\end{equation}
\end{definition}

In the following we will use the following adapted definition (no need to consider extended $ \overline{\mathbb{R}}$).
\begin{definition} \label{def:gen_inverse_modified}
     Let $f$ be a monotonically increasing function defined on the  interval $[a, b]$. The generalized inverse is function $f^{-1}$ defined by
\begin{equation}
    f^{-1}(y)=\inf \{x \in [a, b]: f(x)>y\}, \ y\in (f(a), f(b)) 
\end{equation}
\end{definition}

The following proposition holds:
\begin{proposition}\label{prop:pseudo_inverse_properties} \cite{de2015study}
     The pseudo-inverse $f^{-1}$ of a monotonically increasing function $f$ has following properties:

\begin{enumerate}
    \item $f^{-1}$ is increasing, has left limits and is right continuous.
    \item The following implications hold for all $x \in \mathbb{R}$ and $y \in \mathbb{R}$:
    \begin{align}
        f(x) > y &\implies x \geq f^{-1}(y),\label{eq:prop_geq} \\ 
        f(x) = y &\implies x \leq f^{-1}(y), \\
        f(x) < y &\implies x \leq f^{-1}(y), \label{eq:prop_leq} \\ 
        f^{-1}(y) > x &\implies y \geq f(x), \\
        f^{-1}(y) < x &\implies y < f(x).
    \end{align}
    \item For all $x \in \mathbb{R}$, $f^{-1}(f(x)) \geq x$.
    \item If $f$ is right continuous at $x$, then for all $y \in \mathbb{R}$
    \begin{align}
        f^{-1}(y) = x &\implies y \leq f(x), \\
        f(x) > y &\implies x > f^{-1}(y), \label{eq:prop_geq_repeated}\\
        f(f^{-1}(y)) \geq y.
    \end{align}
    \item $f$ is strictly increasing on $\mathbb{R}$ if, and only if, $f^{-1}$ is continuous on $\mathbb{R}$.
    \item $f$ is continuous on $\mathbb{R}$ if, and only if, $f^{-1}$ is strictly increasing on $\mathbb{R}$.
    \item $f$ is right-continuous if, and only if, $(f^{-1})^{-1} = f$.
    \item $f$ is right-continuous if, and only if, $\{x \in \mathbb{R} : f(x) \geq y\}$ is closed for all $y \in \mathbb{R}$.
\end{enumerate}
\end{proposition}
\iffalse
\begin{proof} I report here only the proofs of \eqref{eq:prop_geq} and \eqref{eq:prop_leq}. In particular, both inequalities are strict if $f$ is continuous. \\
    Proof of \eqref{eq:prop_geq}. From the definition, 
    \begin{equation}
        f(x)> y \implies x \geq \inf \{x \in [a, b]: f(x)>y\}=f^{-1}(y)  .
    \end{equation}
    Furthermore, since $f$ is continuous, 
    \begin{equation}
        f(f^{-1}(y)) = y
    \end{equation}
    and therefore it must hold that $x > f^{-1}(y)$.
        \\
   Proof of \eqref{eq:prop_leq}. From the definition,
    \begin{equation}
        f(x) <  y \implies x \leq \inf \{x \in [a, b]: f(x)>y\} = f^{-1}(y)
    \end{equation}
    Furthermore, since $f$ is continuous, 
    \begin{equation}
        f(f^{-1}(y)) = y
    \end{equation}
        and therefore it must hold that $x < f^{-1}(y)$.

\end{proof}
\fi

\bibliographystyle{ieeetr}
\bibliography{main}

\end{document}